\documentclass[10pt,reqno]{amsart}

\usepackage{amsmath,amssymb,amsfonts,amsthm,eucal,verbatim,fullpage,xcolor}
\usepackage[all]{xy}
\usepackage{hyperref}

\numberwithin{equation}{section}
\newtheorem{theorem}[equation]{Theorem}
\newtheorem{lemma}[equation]{Lemma}
\newtheorem{proposition}[equation]{Proposition}
\newtheorem{corollary}[equation]{Corollary}
\newtheorem{question}[equation]{Question}

\newcommand{\FLEX}{\relax}
\newcommand{\flex}[1]{\renewcommand{\FLEX}{#1}}
\newtheorem{flexthm}[equation]{\FLEX}

\theoremstyle{definition}
\newtheorem{definition}[equation]{Definition}

  \long \def \DavidRemark #1  \endDavidRemark {\color{red}\vskip 6pt \hrule
    \vskip 6pt  \par \noindent \bf
  Remarks for Ruy \& Vrej: \rm #1 \vskip 6pt \hrule \par \vskip 6pt \color{black}}

\long\def\Quiet #1\endQuiet{\relax}

\def\INT{\mathcal I}

\newenvironment{remark}[1]{\refstepcounter{equation}%
\vskip 5pt \par\noindent {\bf #1\ \theequation .}}{\vskip 5pt \par}

\newcommand{\Ad}{\operatorname{Ad}}

\newcommand{\aut}{\operatorname{Aut}}

\newcommand{\bh}{\ensuremath{{\mathcal B}({\mathcal H})}}

\newcommand{\cstar}{\hbox{$C^*$}}
\newcommand{\cstaralg}{$C^*$-algebra}

\newcommand{\dstext}[1]{\quad\text{#1}\quad}
\newcommand{\eps}{\ensuremath{\varepsilon}}

\newcommand{\id}{\text{id}}
\newcommand{\idealin}{\unlhd}
\newcommand{\innerprod}[1]{\left\langle #1\right\rangle}

\newcommand{\norm}[1]{\left\|{#1}\right\|}

\providecommand{\qed}%
{\hfill \vrule height5pt width4pt depth1pt \vspace{+2.00ex}}

\newcommand{\ran}{\operatorname{range}}

\newcommand{\spn}{\operatorname{span}}

\newcommand{\bbC}{{\mathbb{C}}}

\newcommand{\bbE}{{\mathbb{E}}}
\newcommand{\bbF}{{\mathbb{F}}}

\newcommand{\bbT}{{\mathbb{T}}}

\newcommand{\bbZ}{{\mathbb{Z}}}
  \newcommand{\A}{{\mathcal{A}}}
  
  \newcommand{\B}{{\mathcal{B}}}
  \newcommand{\C}{{\mathcal{C}}}
  
  \newcommand{\D}{{\mathcal{D}}}
  
  \newcommand{\E}{{\mathcal{E}}}

\renewcommand{\H}{{\mathcal{H}}}
  \newcommand{\I}{{\mathcal{I}}}
  
  \newcommand{\K}{{\mathcal{K}}}  
\renewcommand{\L}{{\mathcal{L}}}
  \newcommand{\M}{{\mathcal{M}}}
  \newcommand{\N}{{\mathcal{N}}}

\renewcommand{\S}{{\mathcal{S}}}

\newcommand{\fA}{{\mathfrak{A}}}
\newcommand{\fB}{{\mathfrak{B}}}

\newcommand{\nXG}{\text{Norm}(X,\Gamma,\alpha)}
\newcommand{\so}{\text{SO}(3)}

\begin{document}

\title{Exotic Ideals in Free Transformation Group $\mathbf C^*$-Algebras}
\author{Ruy Exel}
\thanks{R. Exel was partially supported by CNPq, Brazil.}
\address{Universidade Federal de Santa Catarina}
\email{exel@mtm.ufsc.br}
\author{David R. Pitts}
\thanks{D.\ Pitts was supported in
  part by a Simons Foundation Collaboration Grant \#316952.}
\address{University of Nebraska--Lincoln}
\email{dpitts2@unl.edu}
\author{Vrej Zarikian}
\address{U. S. Naval Academy}
\email{zarikian@usna.edu}

\subjclass[2010]{46L55}

\begin{abstract}
  Let $\Gamma$ be a discrete group acting freely via homeomorphisms on
  the compact Hausdorff space $X$ and let $C(X) \rtimes_\eta \Gamma$
  be the completion of the convolution algebra $C_c(\Gamma,C(X))$ with
  respect to a \cstar-norm $\eta$. A non-zero ideal
  $J \idealin C(X) \rtimes_\eta \Gamma$ is \emph{exotic} if
  $J \cap C(X) = \{0\}$. We show that exotic ideals are
  present whenever $\Gamma$ is non-amenable and there is an invariant
  probability measure on $X$.  This fact, along with the recent theory
  of exotic crossed product functors, allows us to provide answers to
  two questions of K. Thomsen.

Using the Koopman representation and a recent theorem of Elek, we show
that when $\Gamma$ is a countably-infinite group having property (T)
and $X$ is the Cantor set, there exists a free and minimal action of
$\Gamma$ on $X$ and a \cstar-norm $\eta$ on $C_c(\Gamma, C(X))$ such
that $C(X)\rtimes_\eta\Gamma$ contains the compact operators as an
exotic ideal.   We use this example to provide a positive answer to a question of A. Katavolos and V. Paulsen.

The
opaque and grey ideals in $C(X)\rtimes_\eta \Gamma$ have trivial
intersection with $C(X)$, and a result
from~\cite{ExelPittsChGrC*AlNoHaEtGr} shows they
coincide when the   action of $\Gamma$ is free, however the
problem of whether these ideals can be non-zero was left
unresolved. We present an example of a free action of $\Gamma$
on a compact Hausdorff space $X$ along with a \cstar-norm $\eta$ for
which these ideals are non-trivial, in particular,  they are exotic ideals.
\end{abstract}

\maketitle


\section{Introduction}

An \textit{inclusion} is a pair $(\C,\D)$ of \cstaralg s, with
$\D \subseteq \C$ (we write the larger algebra first!).  We will
always assume that inclusions have the property that $\D$ contains a
unit for $\C$. It is easy to see the equivalence of the following
statements:
\begin{quote} \it
\begin{itemize}
\item[i)] A $*$-representation $\pi$ of $\C$ is faithful if and only if $\pi|_\D$ is faithful.
\item[ii)] If $J \idealin \C$ is an ideal, then $J \cap \D = \{0\}$ if and only if $J = \{0\}$.
\end{itemize}
\end{quote}
Inclusions satisfying either of these conditions are said to have the
\textit{ideal intersection property}, or said to \textit{detect
  ideals}~\cite{KwasniewskiMeyerApToFrPuOu:FrGrAcFeBu},  and a number of important
classes of inclusions possess this property. For example, when $\D$ is
a Cartan MASA in $\C$, $(\C,\D)$ has the ideal intersection
property. The ideal intersection property is of particular interest in
the theory of graph \cstaralg s:  theorems establishing this property
are often called ``uniqueness theorems''.

A non-zero ideal  $J \idealin \C$ is an \textit{exotic ideal} for $(\C,\D)$ if $J \cap \D = \{0\}$. Clearly, $(\C,\D)$ has the ideal intersection property if and only if it has no exotic ideals.

In this note, our primary interest is in exotic ideals for inclusions
$(\C,\D)$ which arise from topological dynamical systems
$(X,\Gamma, \alpha)$; here $X$ is a compact Hausdorff space, $\Gamma$
is a discrete group, and $\Gamma \ni t \mapsto \alpha_t$ is a
homomorphism of $\Gamma$ into the group of homeomorphisms of $X$. We
will refer to $(X,\Gamma,\alpha)$ as a \textit{transformation
  group}. Let $\nXG$ be the collection of \cstar-norms on the
convolution algebra $C_c(\Gamma, C(X))$, and for $\eta \in \nXG$,
denote by $C(X)\rtimes_\eta\Gamma$ the completion of
$C_c(\Gamma, C(X))$ with respect to $\eta$. Fixing $\eta$, we obtain
the inclusion, $(\C,\D) = (C(X) \rtimes_\eta \Gamma,C(X))$.

There has been considerable recent work on what are called ``exotic
crossed products'', see for example,
\cite{BaumGuenterWillettExExCrPrBaCoCo, BussEcterhoffWillettExCrPr,
  BussEcterhoffWillettExCrPrBaCoCo,
  KaliszewskiLandstadQuiggExGrC*AlNoDu}.  A brief discussion of exotic
crossed product functors may be found in Section~\ref{AbExId} below.  Given an
action of $\Gamma$ on the \cstaralg\ $\A$, such
a functor produces a norm $\eta$ on $C_c(\Gamma, \A)$ which lies
between the reduced and full crossed product norms, that is, for
every $h\in C_c(\Gamma,\A)$, $\norm{h}_r\leq \eta(h)\leq \norm{h}_f$.
This implies that there is a $*$-epimorphism $\theta$ from the
completion $\A\rtimes_\eta\Gamma$ of $C_c(\Gamma,\A)$ with respect to
$\eta$ onto the reduced crossed product $\A\rtimes_r\Gamma$.  When
$\eta$ is distinct from the reduced norm, $\ker\theta$ is an ideal having trivial
intersection with $\A$, and hence is an exotic ideal.

For an inclusion $(\C,\D)=(C(X)\rtimes_\eta\Gamma, C(X))$, it is entirely
possible that every non-zero proper ideal of $\C$ is exotic: for
example, this occurs when $\Gamma=\bbZ$, $X$ is a singleton, and
$\eta$ is the full \cstar-norm. Indeed, in that case,
$(\C,\D) \simeq (C(\bbT), \bbC I)$. The problem with this example is
that the action of $\Gamma$ is not free. Because of this, the setting
of most interest to us is when $\Gamma$ acts freely on $X$. When the
action of $\Gamma$ on $X$ is free, $C(X) \rtimes_\eta \Gamma$ is
called a \textit{represented free transformation group}. Represented
free transformation groups were originally introduced by
Thomsen~\cite{ThomsenOnFrTrGpC*Al}, but his definition differs from
the one given here. Remark~\ref{justifyterm} below shows that the two
definitions agree.

Thomsen~\cite[Theorem~9]{ThomsenOnFrTrGpC*Al} shows that when $X$ is
connected and $\Gamma$ acts freely, $(X,\Gamma,\alpha)$ can be
recovered from the inclusion of $C(X)$ in $C(X) \rtimes_\eta
\Gamma$. To do this, Thomsen observes that the transformation groupoid
associated to $(X,\Gamma,\alpha)$ is isomorphic to the Weyl groupoid
of the inclusion; freeness of the action allows the recovery of
$\Gamma$, and $X$ can be recovered from $C(X)$ via Gelfand
theory. Thomsen notes that in some cases, for example when $\Gamma$ is
amenable, there is a unique $\eta\in \nXG$, so that there is a unique
\cstaralg\ associated with $(X,\Gamma,\alpha)$.

When $\Gamma$ acts freely on $X$, the reduced norm on
$C_c(\Gamma,C(X))$ is smallest among all norms in $\nXG$,
\cite[Corollary~7.8]{PittsStReInI}, and the full norm is always the
largest. It follows that there exists a unique $*$-epimorphism
$\Lambda_\eta:C(X) \rtimes_\eta \Gamma \twoheadrightarrow C(X)
\rtimes_r \Gamma$ which restricts to the identity on
$C_c(\Gamma,C(X))$. In turn, this gives rise to a canonical
conditional expectation $E_\eta:C(X) \rtimes_\eta \Gamma \to C(X)$,
namely $E_\eta = E_r \circ \Lambda_\eta$, where
$E_r:C(X) \rtimes_r \Gamma \to C(X)$ is the faithful conditional
expectation such that $E_r(\sum_t a_tt) = a_e$,
$a \in C_c(\Gamma,C(X))$. In \cite[Remark~16]{ThomsenOnFrTrGpC*Al},
Thomsen asks two questions: the first is whether there might be
norms other than the reduced and full norms in $\nXG$, and the second
is whether the conditional expectation $E_\eta$ is faithful for
$\eta \in \nXG$ (equivalently, whether $\Lambda_\eta$ is injective for
$\eta \in \nXG$). When the action of $\Gamma$ is not free, it is
well-known that the full and reduced norms need not coincide, so it
perhaps is unsurprising that Thomsen's second question has a negative
answer. Also, given recent work on exotic crossed products, one
expects a positive answer to Thomsen's first question. We verify these
facts in Section~\ref{ThomsenSect}.

A recent theorem of Elek~\cite[Theorem~1]{ElekFrMiAcCoGrInPrMe} shows
that for any countably-infinite discrete group $\Gamma$, there exists
a transformation group $(X,\Gamma,\alpha)$ with $X$ the Cantor set and
which satisfies the following two additional properties: the action of
$\Gamma$ is both free and minimal, and $(X,\Gamma,\alpha)$ admits an
invariant Borel probability measure
$\mu$. Proposition~\ref{non-atomic} shows that $\mu$ is non-atomic,
and if $\kappa: C(X)\rtimes_f \Gamma\rightarrow \B(L^2(X,\mu))$ is the
Koopman representation, the restriction of $\kappa$ to
$C_c(\Gamma,C(X))$ is one-to-one by
Proposition \ref{Koopman facts}. In particular, the function
$\eta(\cdot):=\norm{\kappa(\cdot)}_{B(L^2(X,\mu))}$ belongs to
$\nXG$. It is well known that it is always possible to choose $\mu$ so
that it is ergodic, and when this is done, the Koopman representation
is irreducible.

In the setting of the previous paragraph, when $\Gamma$ has Kazhdan's
property (T), we show that the compact operators on $L^2(X,\mu)$ are
an exotic ideal in $C(X)\rtimes_\eta\Gamma$, despite the fact that
$\K\cap \kappa(C_c(\Gamma,C(X)))=\{0\}$. It follows from these results
that whenever $\fA$ is a non-atomic MASA in \bh\ and $\fB$ is the
norm-closed linear span of the unitary operators $U\in \bh$ such that
$U\fA U^*=\fA$, then $\fB$ contains the compact operators.  This fact
gives a positive answer to a question of Katavolos and Paulsen which
arose in their study of ranges of bimodule mapping, see
\cite[Remark~13]{KatavolosPaulsenOnRaBiPr}.

Interestingly, several of the results in Section~\ref{FCSDS} are valid
for represented free \cstar-dynamical systems where the \cstaralg\ on
which the group acts is not
assumed abelian.  Because the additional
generality does not take too much extra effort and may be useful, we
give those results in this greater generality.

In section \ref{SecOpId} we discuss \emph{opaque ideals} in free \cstar-dynamical systems: recall from
\cite[Definition 2.10.4]{ExelPittsChGrC*AlNoHaEtGr} that, for any regular inclusion of a commutative \cstar-algebra $\A$
into another \cstar-algebra $\B$, the corresponding opaque ideal is defined by \[
    \Delta = \bigcap_{J \in \widehat \A} J\B,
\]
where $\widehat \A$ is the collection of all maximal ideals in $\A$.  In most nicely behaved examples, the opaque
ideal vanishes, and the
problem of whether or not these ideals can be non-zero was left
unresolved in \cite{ExelPittsChGrC*AlNoHaEtGr}.   Based on a result of Zeller-Meier
\cite[Proposition
5.2]{Zeller-MeierPrCrCstAlGrAu}, stated below as Theorem \ref{amenable char}, we exhibit in
\ref{non-trivial-opaque-ideal} an inclusion arising from a free transformation group \cstar-algebra
whose associated opaque ideal is non-zero.

In the final section, we give some additional applications of the
existence of exotic ideals in represented free transformation groups.

This paper is a substantially revised and reorganized version of an earlier
preprint, entitled, ``Exotic ideals in represented free transformation
groups''.  We thank the referee of that preprint for their comments.   

\section{Preliminaries} \label{prelim}

\subsection{Represented Free Transformation Groups}

Let $(X,\Gamma,\alpha)$ be a \textit{free transformation group}; this means that $X$ is a compact Hausdorff space, $\Gamma$ is a discrete group (with unit $e$), and $\alpha$ is a homomorphism of $\Gamma$ into the homeomorphism group of $X$ which is \emph{free}, that is, for every $e \neq s \in \Gamma$ and $x \in X$, $\alpha_s(x) \neq x$. We will usually write $sx$ instead of $\alpha_s(x)$. The homomorphism $\tau:\Gamma \rightarrow \aut(C(X))$ dual to the action of $\Gamma$ on $X$ is:
\begin{equation} \label{taudef}
   \tau_s(a) = a \circ \alpha_{s^{-1}} \dstext{($a\in C(X)$, $s\in\Gamma$).}
\end{equation}
The transformation group $(X,\Gamma,\alpha)$ and the \cstar-dynamical system $(C(X),\Gamma,\tau)$ are equivalent via Gelfand theory.

Let $C_c(\Gamma,C(X))$ be the collection of finitely supported $C(X)$-valued functions on $\Gamma$. With the usual
convolution product and adjoint operation,
\[
    (ab)(t) := \sum_{s \in \Gamma} a(s)\tau_s(b(s^{-1}t)) \dstext{and} (a)^*(t)=\tau_t(a(t^{-1}))^*,
\]
$C_c(\Gamma,C(X))$ becomes a $*$-algebra. Let  $E:C_c(\Gamma,C(X)) \rightarrow C(X)$ be the map
\begin{equation} \label{Edef}
    E(a)=a(e).
\end{equation}
A computation shows $E$ is faithful in the sense that for $a\in C_c(\Gamma, C(X))$,
\begin{equation}\label{EFaith}
    E(a^*a) = 0 \implies a = 0.
\end{equation}
The restriction of $E$ to the subalgebra
\[
    \{a \in C_c(\Gamma, C(X)): a(t) = 0 \text{ for all } t \neq e\}
\]
is an isomorphism onto $C(X)$. We shall identify $f \in C(X)$ with the function in $C_c(\Gamma,C(X))$ taking the value $f$ at $e$ and $0$ at every other $s \in \Gamma$. In this way, we regard $C(X)\subseteq C_c(\Gamma,C(X))$.

For $s \in \Gamma$, let $\delta_s \in C_c(\Gamma, C(X))$ be the function
\[
    \delta_s(t) = \begin{cases}
        {\mathbf 1} & \text{if $t = s$;}\\
        0 & \text{if $t \neq s$,}
    \end{cases}
\]
where $\mathbf 1$ denotes the multiplicative identity in $C(X)$. Then $\delta_s\delta_t=\delta_{st}$ and $\delta_{s^{-1}} = \delta_s^* = (\delta_s)^{-1}$. We may therefore identify $\Gamma$ with the group $\{\delta_s: s \in \Gamma\}$, and this is what is meant when we write $\Gamma\subseteq C_c(\Gamma,C(X))$.

Let $\nXG$ be the family of all \cstar-norms on $C_c(\Gamma,C(X))$. For $\eta \in \nXG$ we denote the completion of
$C_c(\Gamma,C(X))$ with respect to $\eta$ by $C(X) \rtimes_\eta \Gamma$. We remark  that for $f \in C(X)$, uniqueness of
\cstar-norms gives $\eta(f) = \norm{f}_{C(X)}$.


When $\eta$ is the reduced or full \cstar-norm, we will replace the subscript $\eta$ in the crossed product notation with $r$ or $f$ respectively; also we use $\norm{\cdot}_r$ and $\norm{\cdot}_f$ for the reduced and full \cstar-norms. The map $E:C_c(\Gamma,C(X)) \to C(X)$ from \ref{Edef} extends uniquely to a faithful conditional expectation $E_r:C(X) \rtimes_r \Gamma \to C(X)$ (see \cite[Proposition~4.1.9]{BrownOzawaC*AlFiDiAp}).

\begin{definition} \label{rftg} A \textit{represented free transformation group} is a \cstaralg\ having the form \[C(X)\rtimes_\eta\Gamma\] for some freely acting discrete dynamical system $(X,\Gamma,\alpha)$ and $\eta\in \nXG$.

An ideal $J\idealin C(X)\rtimes_\eta \Gamma$ is \textit{exotic} if $J\neq \{0\}$ and $C(X)\cap J=\{0\}$.
\end{definition}

\begin{remark}{Remarks} \label{justifyterm}
  \begin{enumerate}
\item In our discussion of Thomsen's problems in the introduction and
  in Definition~\ref{rftg}, we have paraphrased Thomsen's setting and
  we discuss that now. The term ``represented free transformation
  group'' was originally coined by K. Thomsen, which we reproduce here
  for the convenience of the reader, see
  \cite[Definition~3]{ThomsenOnFrTrGpC*Al}.  Thomsen defines a
  represented free transformation group to be a triple $(B,A,S)$, with
  $B$ a unital \cstaralg\ containing the abelian \cstar-subalgebra $A$
  and $S$ a group of unitaries in $B$ such that: a) for $u\in S$, $uAu^*=A$; b) for
  $I\neq u\in S$,
  $(\Ad u)|_A$ is a free $*$-automorphism of $A$, i.e.\ the
  dual homeomorphism on $\widehat A$ is free; and c) $B=C^*(A\cup S)$.

  While our definition differs from his, the two definitions are equivalent. We sketch a proof of the equivalence. Note that given $\eta \in \nXG$, we have
  \[
        C(X) \subseteq C_c(\Gamma,C(X)) \subseteq C(X) \rtimes_\eta \Gamma \dstext{and} \Gamma \subseteq C_c(\Gamma,C(X)) \subseteq C(X) \rtimes_\eta \Gamma.
    \]
    Then the triple, $(C(X)\rtimes_\eta \Gamma,C(X),\Gamma)$ is a represented free transformation group in Thomsen's sense.

    On the other hand, suppose $(B,A,S)$ is a represented free transformation group in Thomsen's sense. Set $X = \widehat A$, and identify $A$ with $C(X)$ via the Gelfand transform. By \cite[Lemma~11]{ThomsenOnFrTrGpC*Al} there is a conditional expectation $\bbE:B \rightarrow A$ satisfying $\bbE(s) = 0$ for any $e \neq s \in S$. Thus the map $\theta:C_c(S,C(X)) \rightarrow B$ given by
    \[
        \theta(f) = \sum_{s \in S} f(s)s
    \]
    is a $*$-monomorphism with dense range. Therefore, we may define a \cstar-norm on $C_c(S,C(X))$ by $\eta(f) = \norm{\theta(f)}_B$, so that
    \[
        C(X) \rtimes_\eta S \simeq B.
    \]
    In this way, we may regard $(B,A,S)$ as a represented free transformation group in the sense of Definition~\ref{rftg} above.
  \item
When the norm $\eta\in \nXG$ arises from an exotic crossed product,
the kernel of the map $C(X)\rtimes_\eta\Gamma\rightarrow
C(X)\rtimes_r\Gamma$ is an exotic ideal for $(C(X)\rtimes_\eta\Gamma,
C(X)$.   It is this fact which led us to the terminology ``exotic ideal'' found in Definition~\ref{rftg}

  \end{enumerate}
\end{remark}

Returning to the context of free transformation groups, there is no
exotic ideal in $C(X)\rtimes_r\Gamma$.  This follows from
\cite[Theorem 1]{ArchboldSpielbergToFrAcIdDiC*DySy}, but it can also
be deduced from Corollary~\ref{free facts}(f) below and the fact that
$E_r$ is always faithful.  Regarding $C(X)\rtimes_r\Gamma$ as the
reduced \cstaralg\ of the transformation groupoid associated to
$(X,\Gamma,\alpha)$, this fact may further be obtained
from~\cite[Theorem~4.4]{ExelNoHaEtGp}.  Also, there is no exotic ideal
when $\Gamma$ is amenable, since
$C(X)\rtimes_\eta\Gamma = C(X)\rtimes_r\Gamma$ in that case.

Represented free transformation groups are examples of regular inclusions. We pause to establish some conventions and notation which will be useful in the sequel.

\begin{definition} \label{various} Let $(\B,\A)$ be an inclusion.
\begin{enumerate}
\item A \textit{normalizer of $\A$} is an element of the set
\[
    \N(\B,\A) = \{v \in \B: v\A v^* \cup v^*\A v \subseteq \A\}.
\]
Then $\N(\B,\A)$ is a $*$-semigroup and the linear span of $\N(\B,\A)$ is a $*$-algebra.
\item The inclusion $(\B,\A)$ is \textit{regular} if $\spn\N(\B,\A)$ is dense in $\B$ and is \textit{singular} when
$\N(\B,\A) = \A$.

\item An \textit{intertwiner of $\A$} is an element $v \in \B$ such that $v\A = \A v$; we use $\INT(\B,\A)$ for the set of all intertwiners.
\item When $\A$ is a maximal abelian $*$-subalgebra (MASA) in $\B$, we say $(\B,\A)$ is a \textit{MASA inclusion}.
\item When every pure state of $\A$ uniquely extends to a pure state of $\B$, we say that $(\B,\A)$ has the \textit{extension property}, or that $(\B,\A)$ is an \textit{EP-inclusion}. When $(\B,\A)$ is an EP-inclusion and $\A$ is abelian, $\A$ is a MASA in $\B$ and there exists a unique conditional expectation of $\B$ onto $\A$~\cite[Corollary~2.7]{ArchboldBunceGregsonExStC*AlII}.
\item As with represented free transformation groups, we call a non-zero ideal $J \idealin \B$
  \textit{exotic} if $J \cap \A = \{0\}$.
\end{enumerate}
\end{definition}

The notion of pseudo-expectations (see~\cite{PittsStReInI,
  PittsZarikianUnPsExC*In})  will be useful for our work.
Hamana~\cite{HamanaInEnC*Al} showed that given a \textit{unital}
\cstar-algebra $\A$ there is a \cstar-algebra $I(\A)$ and a one-to-one
unital $*$-homomorphism $\iota :\A \rightarrow I(\A)$ such that:
\begin{itemize}
  \item $I(\A)$ is an injective object in the category of operator systems and
    unital completely positive  maps; and
  \item if $\S$ is an injective operator system with
    $\iota(\A)\subseteq \S\subseteq I(\A)$, then $\S=I(\A)$.
\end{itemize}
The pair $(I(\A),\iota)$ is called an \emph{injective
  envelope} of $\A$.  The injective envelope of $\A$ is monotone closed and has
the following uniqueness property: if $(\I_1,\iota_1)$ and
$(\I_2,\iota_2)$ are injective envelopes for $\A$, there exists a
unique $*$-isomorphism $\alpha:\I_1\rightarrow \I_2$ such that
$\alpha\circ \iota_1=\iota_2\circ \alpha$.
When the embedding $\iota$ is fixed for a discussion and there is no
danger of confusion, we will usually identify  $\A$ with $\iota(\A)$;
when this is done, we view $\A$ as a subalgebra of $I(\A)$.
\begin{definition}   Let $(\B,\A)$ be an inclusion.   A
  \textit{pseudo-expectation} for $(\B,\A)$ is a unital
  completely positive map $E:\B\rightarrow I(\A)$ such that
  $E|_\A=\id|_\A$.
\end{definition}
When $\Delta:\B\rightarrow \A$ is a conditional expectation, $\Delta$
is necessarily a pseudo-expectation, so the notion of
pseudo-expectation extends the notion of conditional expectation.
While a conditional expectation need not exist, due to the injectivity
of $I(\A)$, pseudo-expectations always exist.  In general, a given
inclusion has many pseudo-expectations. However, in some important
cases, there is a unique pseudo-expectation.  When this occurs, we say
that $(\B,\A)$ has the \textit{unique pseudo-expectation property}.
To say the inclusion $(\B,\A)$ has the \textit{faithful unique
  pseudo-expectation property} means that $(\B,\A)$ has the unique
pseudo-expectation property and the pseudo-expectation is faithful.
\subsection{Free $C^*$-Dynamical Systems}\label{FCSDS}

Although our main interest is in represented free transformation
groups, many of the basic
facts we will need are true in the greater generality of represented free $C^*$-dynamical systems.

Let $(\A,\Gamma,\tau)$ be a discrete $C^*$-dynamical system, with
$\A$ not assumed abelian. By definition, $\|\cdot\|_f$ is the largest
$C^*$-norm on $C_c(\Gamma,\A)$. In general, there is no smallest
$C^*$-norm on $C_c(\Gamma,\A)$ \cite[Example 7.1]{PittsStReInI},
however there are circumstances when there are minimal \cstar-norms on
certain $*$-subalgebras arising from a regular inclusion.  Some of
these are discussed in~\cite[Section~7]{PittsStReInI}.
Our next goal
is to show that if
$\Gamma \curvearrowright \widehat{\A}$ is free, then $\|\cdot\|_r$ is
the smallest $C^*$-norm on $C_c(\Gamma,\A)$, a fact which is critical
to our analysis.  We remind the reader of some of the relevant notions.

If $\A$ is a \cstaralg, its spectrum, denoted~$\widehat\A$, is the
collection of all unitary equivalence classes of non-zero
$*$-representations of $\A$.  The map $[\pi]\mapsto \ker\pi$ is a
surjection of $\widehat\A$ onto $\text{Prim}(\A)$, the primitive ideal
space, of $\A$, and the topology on $\widehat\A$ is the smallest topology
which makes this map continuous when $\text{Prim}(\A)$ is equipped
with the hull-kernel topology.  When $(\A,\Gamma,\tau)$ is a discrete
\cstar-dynamical system, there is an associated action $\alpha$ of
$\Gamma$ on $\widehat\A$:
\[\alpha_s([\pi])=[\pi\circ\tau_s^{-1}].\]  This action is free when
$e\neq s\implies \alpha_s$ has no fixed points.


A sufficient condition (weaker than freeness of
$\Gamma\curvearrowright\widehat\A$)  for minimality of the reduced norm is that
$\Gamma \curvearrowright \A$ is properly outer in Kishimoto's sense
\cite[Section 2]{KishimotoFrAcAuC*Al}.   By~\cite[Corollary 3.6]{ZarikianUnCoExAbC*In},
the action of $\Gamma$ on $\A$ is Kishimoto properly outer if and only if
the inclusion $(\A \rtimes_f \Gamma, \A)$ has the unique
pseudo-expectation property.   Since we wish to highlight the role of
pseudo-expectations, we state the following result in terms of the unique
pseudo-expectation property instead of using Kishimoto's properly
outer condition.

Before stating the next result, we pause to establish some notation
which we use several times in the sequel.  There is a unique $*$-epimorphism
$\Lambda_f:\A \rtimes_f \Gamma \twoheadrightarrow \A \rtimes_r \Gamma$
which extends the identity map on
$C_c(\Gamma,\A)$; we will call $\Lambda_f$ the \textit{canonical
  $*$-epimorphism}.

We now show that the reduced norm is the smallest \cstar-seminorm on
$C_c(\Gamma,\A)$ whose restriction to $\A$ coincides with the norm on $\A$.
\begin{theorem} \label{smallest norm} Let $(\A,\Gamma,\tau)$ be a
  discrete $C^*$-dynamical system and suppose the inclusion $(\A\rtimes_f\Gamma,\A)$ has the
  unique pseudo expectation property (in particular, this occurs when
  $\Gamma \curvearrowright \widehat{\A}$ is free).   If $\sigma$ is a
  \cstar-seminorm such that for every $a\in \A$, $\sigma(a)=\norm{a}_\A$, then
  for every $x\in C_c(\Gamma,\A)$,
  \[\norm{x}_r\leq \sigma(a).\]
\end{theorem}

\begin{proof}
First note that the unique pseudo-expectation for $(\A \rtimes_f
\Gamma,\A)$ is $E_f = E_r \circ \Lambda_f$, where $E_r:\A \rtimes_r
\Gamma \to \A$ is the canonical faithful conditional expectation.
A straightforward computation shows that $\L_f =
\ker(\Lambda_f)$, where $\L_f = \{x \in \A \rtimes_f \Gamma: E_f(x^*x)
= 0\}$ is the left kernel of $E_f$. Suppose that $J \idealin \A
\rtimes_f \Gamma$ and $J \cap \A = \{0\}$. Then the map $\theta_0:\A+J
\to \A: a+j \mapsto a$ is a well-defined $*$-homomorphism such that
$\theta_0|_{\A} = \id$. Let $\theta:\A \rtimes_f \Gamma \to I(\A)$ be
a UCP (unital completely positive) extension of $\theta_0$. Then $\theta$ is a pseudo-expectation for $(\A \rtimes_f \Gamma,\A)$, and so $\theta = E_f$, by uniqueness. Thus for $j \in J$, we have that
\[
    E_f(j^*j) = \theta(j^*j) = \theta_0(j^*j) = 0,
\]
which implies $J \subseteq \ker(\Lambda_f)$. Now let $\sigma$ be a
$C^*$-seminorm on $C_c(\Gamma,\A)$, let $N\idealin C_c(\Gamma,\A)$ be
the ideal $N=\{x\in C_c(\Gamma,\A): \sigma(x)=0\}$, and let
$\A \rtimes_\sigma \Gamma$ denote the completion of $C_c(\Gamma,\A)/N$
with respect to $\sigma$. Since $\|\cdot\|_f$ is the largest
$C^*$-norm on $C_c(\Gamma,\A)$, there exists a unique $*$-epimorphism
$\pi:\A \rtimes_f \Gamma \twoheadrightarrow \A \rtimes_\sigma \Gamma$
which extends the quotient map of $C_c(\Gamma,\A)$ onto $C_c(\Gamma,\A)/N$. Since
$\ker(\pi) \cap \A = \{0\}$, the above argument shows that
$\ker(\pi) \subseteq \ker(\Lambda_f)$. Thus for
$x \in C_c(\Gamma,\A)$, we have that
\[
    \|x\|_r = \|x + \ker(\Lambda_f)\| \leq \|x + \ker(\pi)\| = \sigma(x).
\]
\end{proof}

Some of the statements (e.g.\ items \eqref{ff6} and \eqref{ff7}) in the
following corollary extend results of~\cite{PittsStReInI} to the
context of freely acting discrete \cstar-dynamical systems.

\begin{corollary} \label{free facts}
Let $(\A,\Gamma,\tau)$ be a discrete $C^*$-dynamical system. If $\Gamma \curvearrowright \widehat{\A}$ is free and $\eta$ is a $C^*$-norm on $C_c(\Gamma,\A)$, then the following statements hold:
\begin{enumerate}
\item\label{ff1} Every pure state on $\A$ extends uniquely to a pure state on $\A \rtimes_\eta \Gamma$.
\item\label{ff2} $(\A \rtimes_\eta \Gamma,\A)$ is a regular inclusion and $\A' \cap (\A \rtimes_\eta \Gamma) = Z(\A)$ (in particular, if $\A$ is abelian, then it is a MASA in $\A \rtimes_\eta \Gamma$).
\item\label{ff3} For every $x \in C_c(\Gamma,\A)$, $\|x\|_r \leq \eta(x) \leq \|x\|_f$.
\item\label{ff4} There exists a unique $*$-epimorphism $\Lambda_\eta:\A \rtimes_\eta \Gamma \twoheadrightarrow \A \rtimes_r \Gamma$ extending the identity map on $C_c(\Gamma,\A)$.
\item\label{ff5} $(\A \rtimes_\eta \Gamma,\A)$ has a unique pseudo-expectation, namely $E_\eta = E_r \circ \Lambda_\eta$.
\item\label{ff6} The ``left kernel'' of $E_\eta$, i.e., the set $\L_\eta := \{x \in \A \rtimes_\eta \Gamma: E_\eta(x^*x) =0\}$, is a closed, two-sided ideal of $\A \rtimes_\eta \Gamma$. Indeed, $\L_\eta = \ker(\Lambda_\eta)$. Obviously $\L_\eta \cap \A = \{0\}$. In fact, $\L_\eta$ has the following maximality property: If $J \idealin \A \rtimes_\eta \Gamma$ and $J \cap \A = \{0\}$, then $J \subseteq \L_\eta$. In particular, if $\A \rtimes_\eta \Gamma$ has an exotic ideal, then $\L_\eta$ is the largest exotic ideal.
\item\label{ff7} If $J \idealin \A \rtimes_\eta \Gamma$ and $J \cap \A = \{0\}$, then $J \cap C_c(\Gamma,\A) = \{0\}$.
\end{enumerate}
\end{corollary}

\begin{proof}
(a) By \cite[Corollary 2.5]{ZarikianPuExPrDiCrPr}, the inclusion $(\A \rtimes_f \Gamma,\A)$ has the extension property (see also \cite[Theorem 2]{AkemannWeaverCoCoNaPr}). Since the identity map $\id:C_c(\Gamma,\A) \to C_c(\Gamma,\A)$ extends uniquely to a $*$-epimorphism $\pi:\A \rtimes_f \Gamma \twoheadrightarrow \A \rtimes_\eta \Gamma$, the inclusion $(\A \rtimes_\eta \Gamma,\A)$ has the extension property \cite[Lemma 3.1]{ArchboldBunceGregsonExStC*AlII}.

(b) Regularity is immediate; the fact that the relative commutant equals the center follows from \cite[Theorem 3.1]{BunceChuUnExPuStC*Al}.

(c) Theorem \ref{smallest norm}.

(d) This follows immediately from (c).

(e) Since $(\A \rtimes_\eta \Gamma,\A)$ has the extension property, it has a unique pseudo-expectation, necessarily $E_\eta = E_r \circ \Lambda_\eta$.

(f) This follows as in the proof of Theorem \ref{smallest norm}.

(g) Suppose $J \idealin \A \rtimes_\eta \Gamma$ and $J \cap \A = \{0\}$. By (f), $J \subseteq \L_\eta = \ker(\Lambda_\eta)$. Thus
\[
    J \cap C_c(\Gamma,\A) \subseteq \ker(\Lambda_\eta) \cap C_c(\Gamma,\A) = \{0\}.
\]
\end{proof}

Let $(\A,\Gamma,\tau)$ be a discrete $C^*$-dynamical system. If
$\Gamma$ is amenable, then the full and reduced norms on
$C_c(\Gamma,\A)$ coincide \cite[Theorem
4.2.6]{BrownOzawaC*AlFiDiAp}. If $(\A,\Gamma,\tau)$ admits an
invariant state $\phi$, that is $\phi\circ\tau_s=\phi$ for every
$s\in\Gamma$,  then the converse is true \cite[Proposition
5.2]{Zeller-MeierPrCrCstAlGrAu}. We provide a short proof for the
reader's convenience---we thank Rufus Willett and Alcides Buss for
showing us the following argument.

\begin{theorem} \label{amenable char}
Let $(\A,\Gamma,\tau)$ be a discrete $C^*$-dynamical system which admits an invariant state $\phi$. If the full and reduced norms on $C_c(\Gamma,\A)$ coincide, then $\Gamma$ is amenable.
\end{theorem}

\begin{proof}
Consider the GNS representation $(\pi_\phi,\H_\phi,\xi_\phi)$ of $\A$ with respect to the invariant state $\phi$. For each $s \in \Gamma$ and $a \in \A$, define $U_s\pi_\phi(a)\xi_\phi = \pi_\phi(\tau_s(a))\xi_\phi$. Then $U_s:\pi_\phi(\A)\xi_\phi \to \pi_\phi(\A)\xi_\phi$ is a well-defined linear isometry, since
\[
  \|\pi_\phi(\tau_s(a))\xi_\phi\|^2 = \phi(\tau_s(a^*a)) =
  \phi(a^*a) = \|\pi_\phi(a)\xi_\phi\|^2.
\]
Thus $U_s$ extends uniquely to an isometry $\H_\phi \to \H_\phi$, again denoted by $U_s$. Since $U_e = I$ and $U_sU_t = U_{st}$ for all $s, t \in \Gamma$, the map $\Gamma \to B(\H_\phi): s \mapsto U_s$ is a unitary representation. Because of the covariance condition $U_s\pi_\phi(a)U_s^* = \pi_\phi(\tau_s(a))$, there exists a unique representation $\kappa_\phi:\A \rtimes_f \Gamma \to B(\H_\phi)$ such that $\kappa_\phi(a) = \pi_\phi(a)$, $a \in \A$, and $\kappa_\phi(s) = U_s$, $s \in \Gamma$. If $\omega_\phi$ denotes the vector state on $B(\H_\phi)$ associated to $\xi_\phi \in \H_\phi$, then
\[
  \omega_\phi(\kappa_\phi(s)) = \omega_\phi(U_s) = \langle
  U_s\xi_\phi, \xi_\phi \rangle = \langle \xi_\phi, \xi_\phi \rangle =
  1, ~ s \in \Gamma.
\]
Now assume that the full and reduced norms on $C_c(\Gamma,\A)$ coincide. If $\iota:C_r^*(\Gamma) \to \A \rtimes_r \Gamma$ denotes the canonical inclusion, then the composition
\[
\xymatrix{
    C_r^*(\Gamma) \ar[r]^\iota \ar@/^3pc/[rrrr]^\psi & \A \rtimes_r \Gamma \ar[r]^\id & \A \rtimes_f \Gamma\ar[r]^{\kappa_\phi} & \B(\H_\phi)\ar[r]^{\omega_\phi} & \bbC
}
\]
defines a character $\psi$ on $C_r^*(\Gamma)$, which implies the amenability of $\Gamma$ \cite[Theorem 2.6.8]{BrownOzawaC*AlFiDiAp}.
\end{proof}

\begin{remark}{Remark} \label{Koopman} Our proof of Theorem
  \ref{amenable char} made use of a particular representation
  $\kappa_\phi:\A \rtimes_f \Gamma \to B(\H_\phi)$ arising from an
  invariant state $\phi$ for $(\A,\Gamma,\tau)$. We will refer to this
  representation of $\A\rtimes_f\Gamma$ as the \emph{Koopman
    representation}, denoted $\kappa_\phi$. That is,
  $\kappa_\phi(a) = \pi_\phi(a)$, $a \in \A$, and
  $\kappa_\phi(s) = U_s$, $s \in \Gamma$. Our choice of terminology is
  justified by the abelian case. Indeed, if $\A = C(X)$ for a compact
  Hausdorff space $X$, $\Gamma$ is a discrete group acting on $X$ by
  homeomorphisms, and $\phi(f) = \int_X f(x)d\mu(x)$ for an invariant
  regular Borel probability measure $\mu$ on $X$, then
  $\H_\phi = L^2(X,\mu)$, $\kappa_\phi(f) = M_f$ (the operator on
  $L^2(X,\mu)$ of multiplication by $f$) for all $f \in C(X)$, and
  $\kappa_\phi(s)\xi = \xi \circ s^{-1}$ (the so-called Koopman
  operator on $L^2(X,\mu)$ corresponding to $s$) for all
  $s \in \Gamma$.
\end{remark}

The following proposition identifies sufficient conditions for the
restriction to $C_c(\Gamma,\A)$ of the
Koopman representation to be one-to-one or irreducible. Recall that an
action $G \curvearrowright X$ of a discrete group on a topological
space by homeomorphisms is \emph{minimal} if the only closed invariant
subsets are $\emptyset$ and $X$, and that an invariant state $\phi$
for a discrete $C^*$-dynamical system $(\A,\Gamma,\tau)$ is
\emph{ergodic} \cite[Section~4.3.1]{BratteliRobinsonOpAlQuStMe1} if it
is extremal among all invariant states.

\begin{proposition} \label{Koopman facts}
Let $\phi$ be an invariant state for the discrete $C^*$-dynamical system $(\A,\Gamma,\tau)$ and $\kappa_\phi:\A\rtimes_f\Gamma \to B(\H_\phi)$ be the corresponding Koopman representation.
\begin{enumerate}
\item[(a)] If $\Gamma \curvearrowright \widehat{\A}$ is free and
  minimal, then the restriction of $\kappa_\phi$ to $C_c(\Gamma,\A)$ is one-to-one.
\item[(b)] If $\phi$ is ergodic, then $\kappa_\phi$ is irreducible.
\end{enumerate}
\end{proposition}

\begin{proof}
  (a) Since $\Gamma \curvearrowright \widehat{\A}$ is minimal, $\A$
  has no nontrivial (closed, 2-sided) $\Gamma$-invariant ideals
  \cite[Proposition 3.2.2]{DixmierC*Al}. Because of the covariance
  relation $U_s\pi_\phi(a)U_s^* = \pi_\phi(\tau_s(a))$, we see that
  $\ker\pi_\phi \idealin \A$ is $\Gamma$-invariant, which implies
  that $\ker(\pi_\phi) = \{0\}$, which in turn implies that
  $\ker(\kappa_\phi) \cap \A = \{0\}$. Since
  $\Gamma \curvearrowright \widehat{\A}$ is free,
  $\ker(\kappa_\phi) \cap C_c(\Gamma,\A) = \{0\}$, by Corollary
  \ref{free facts}, part (g). That is, $\kappa_\phi$ is one-to-one on $C_c(\Gamma,\A)$.

(b) \cite[Theorem 4.3.17]{BratteliRobinsonOpAlQuStMe1}.
\end{proof}

Let $\phi$ be an invariant state for the discrete $C^*$-dynamical
system $(\A,\Gamma,\tau)$. We would like to know when
$\kappa_\phi(\A\rtimes_f\Gamma)$ contains $\K(\H_\phi)$, the
compact operators on $\H_\phi$. In light of Proposition \ref{Koopman
  facts}, part (b), in makes sense to focus on the situation when
$\phi$ is ergodic, in which case, by a standard fact about irreducible
representations of \cstaralg s
(see for example~\cite{ArvesonAnInC*Al}),  it is enough to show that
$\kappa_\phi(\A\rtimes_f\Gamma)$ contains a nonzero compact
operator.

Our next result shows that sufficient hypotheses for
$\kappa_\phi(\A\rtimes_f\Gamma)\cap \K(\H_\phi)\neq \{0\}$ are that $\A$ is abelian and $\Gamma$ is a countable group
with Kazhdan's property (T). The proof is essentially due to Chou,
Lau, and Rosenblatt \cite[Theorem 1.1, (b) $\Rightarrow$
(a)]{ChouLauRosenblattApCoOpSuTr}. We are indebted to Bill Johnson for
alerting us to this reference.

\begin{theorem} \label{contains compacts} Let $X$ be a compact
  Hausdorff space, and $\Gamma$ be a countable group with property
  (T).  Suppose $\phi$ is an ergodic state for the discrete
  $C^*$-dynamical system $(C(X),\Gamma,\tau)$, and let
  $\kappa_\phi$ be the corresponding
  Koopman representation.  Then
  $\K(\H_\phi) \subseteq \kappa_\phi(C(X)\rtimes_f\Gamma)$.
\end{theorem}

\begin{proof}

  Let $\xi_\phi\in \H_\phi$ be the cyclic vector for $\kappa_\phi(C(X))$
  in arising from the GNS representation, that is,
  $\xi_\phi=I_{C(X)}+\N$, where $\N$ is the left kernel of $\phi$.  Let
  $P\in\B(\H_\phi)$ be the rank-one projection whose range is
  $\bbC\xi_\phi$ and set \[\C:=\kappa_\phi(C(X)\rtimes_f\Gamma).\]
  Then $\C\subseteq \B(\H_\phi)$ is irreducible by ergodicity (see
  Proposition~\ref{Koopman facts}(b)).
  Thus, to show $\C$ contains the set of compact operators, it
  suffices to prove $P\in\C$.

Let $\M$ be the norm-closed linear span of
  $\{U_s\}_{s\in\Gamma}$.  As $\M\subseteq \C$, it is enough to show
  $P\in \M$.  Note that $\H_0:=(I-P)\H_\phi$ is invariant for each
  $U_s$ so that the map $s\mapsto V_s$ is a unitary representation of
  $\Gamma$ on $\H_0$ where
  \[V_s:=U_s|_{\H_0}.\]  Since $\phi$ is an ergodic state, \cite[Theorem
  4.3.17]{BratteliRobinsonOpAlQuStMe1} shows
  $\ran P=\{\xi\in\H_\phi: U_s\xi=\xi\text{ for all } s\in \Gamma\}$
  and no vector in $\H_0$ is
  invariant under each $V_s$.  Thus, since $\Gamma$  has property (T),
  there exists $\eps>0$ and $s_1,\dots, s_n\in \Gamma$  such that for
  all $\xi\in \H_0$ with $\norm{\xi}=1$,
  \begin{equation}\label{noai} \max_{1\leq i\leq n}
    \norm{V_{s_i}\xi-\xi}\geq \eps.
  \end{equation}
  Set $s_0=e$ and let
  \[A:=\frac{1}{n+1} \sum_{j=0}^n U_{s_j}.\]    Obviously,
  $\norm{A}\leq 1$.   We show that
  $\norm{A|_{\H_0}}<1$.  If not, there is a sequence $\eta_k$ of unit
    vectors in $\H_0$ with $\norm{A\eta_k}\rightarrow 1$.    This means
    that  \[\lim_{k\rightarrow \infty} \frac{1}{(n+1)^2}\sum_{i,j=0}^n
      \innerprod{V_{s_i}\eta_k,V_{s_j}\eta_k} =1.\]   As
    $|\innerprod{V_{s_i}\eta_k,V_{s_j}\eta_k}|\leq 1$ for $i,j, k$ and
    $1$ is an extreme point of the unit disc, we
    conclude that for each $0\leq i,j\leq n$,
    \[\lim_k \innerprod{V_{s_i}\eta_k,V_{s_j}\eta_k}=1.\]  Therefore,
    for $1\leq i\leq n$, we obtain
    $\lim_{k\rightarrow\infty}\norm{V_{s_i}\eta_k-\eta_k}=0$,  contradicting~\eqref{noai}.

    As $\norm{A|_{\H_0}}<1$, $A^n\rightarrow P$.  So $P\in \M$, as
    desired.
\end{proof}

\subsection{A Little Measure Theory}

We will also need a few (almost-surely known) facts about regular Borel measures. We provide the proofs for the reader's convenience. We note that in many cases of interest, for example when $X$ is a compact metric space, every Borel measure on $X$ is regular---see~\cite[Sec.\ 21.5, Theorem~14]{RoydenReAn} (recall that regular Borel measures are also called Radon measures~\cite[footnote, p.~456]{RoydenReAn}).

\begin{lemma}\label{RuyMeasure}
Let $\mu $ be a finite, regular Borel measure on a locally compact Hausdorff space $X$. If the Borel set $A \subseteq X$ is an atom for $\mu$, then there is a point $x_0 \in A$ such that $A = \{x_0\} \cup N$, and $\mu(N) = 0$.
\end{lemma}

\begin{proof}
By regularity, there exists a compact set $K \subseteq A$, with $\mu(K) > \mu(A)/2$. Since $A$ is an atom, it follows    that $\mu(K) = \mu(A)$, whence $K$ is also an atom. So we may assume without loss of generality that $K = A$, and hence that $A$ is a compact subset of $X$.

Let $\nu$ be the restriction of $\mu$ to a measure on $A$, and let $S$ be the support of $\nu$, that is, $S$ is the smallest closed subset of $A$ with $\nu(A \setminus S) = 0$. It follows that $\nu(S) = \nu(A)$, whence $S$ is also an atom. So we may assume without loss of generality that $S=A$, and hence that $\nu$ has full support.

We next claim that $A$ is a singleton. Arguing by contradiction, suppose that $A$ has more than one point. We may then write $A$ as the union of two proper closed subsets $C_1$ and $C_2$ (e.g.~the complements of open neighborhoods separating two points in $A$). Since $\nu$ has full support it follows that $\nu(C_i) < \nu (A)$, so $\nu(C_i) = 0$, because $A$ is an atom. Consequently $\nu(A) \leq \nu(C_1) + \nu(C_2) = 0$, a contradiction.
\end{proof}

\begin{proposition} \label{non-atomic}
Let $(X,\Gamma,\alpha)$ be a free transformation group and $\mu$ be an invariant regular Borel probability measure. If $\Gamma$ is infinite, then $\mu$ is non-atomic.
\end{proposition}

\begin{proof}
Suppose that $\Gamma$ is countably infinite and $\mu$ has an atom. By Lemma \ref{RuyMeasure}, there exists $x_0 \in X$ such that $\mu(\{x_0\}) > 0$. By freeness, $sx_0 \neq tx_0$ whenever $s, t \in \Gamma$ and $s \neq t$. Thus
\[
    \mu(\{sx_0: s \in \Gamma\}) = \sum_{s \in \Gamma} \mu(\{sx_0\}) = \sum_{s \in \Gamma} \mu(\{x_0\}) = \infty,
\]
a contradiction.
\end{proof}

\begin{remark}{Remark}\label{erinex}
Let $M(X,\Gamma)$ be the collection of all invariant Borel probability measures for $(X,\Gamma,\alpha)$, and assume $M(X,\Gamma)$ is not empty. Viewing each $\mu\in M(X,\Gamma)$ as a state on $C(X)$, we see that $M(X,\Gamma)$ is convex and weak-$*$ closed, hence weak-$*$ compact. The Krein-Milman theorem assures the existence of an extreme point $\mu\in M(X,\Gamma)$. A standard argument shows $\mu$ is ergodic (\cite[Theorem~6.10(iii)]{WaltersInErTh} gives the argument when $\Gamma=\bbZ$, but the proof applies for general $\Gamma$).
\end{remark}

\section{Main Results}

In this section, we establish the existence of exotic ideals in certain represented free transformation groups. In addition to the results from Section~\ref{prelim}, our arguments will rely on a recent theorem of Elek, which shows that every countably-infinite, non-amenable discrete group $\Gamma$ has a free and minimal action on the Cantor set $X$ which admits an invariant (regular) Borel probability measure $\mu$ \cite[Theorem~1]{ElekFrMiAcCoGrInPrMe}. By Proposition \ref{non-atomic}), $\mu$ is non-atomic. By Remark \ref{erinex}, we may assume that $\mu$ is ergodic. We are grateful to Eduardo Scarparo for alerting us to Elek's paper \cite{ElekFrMiAcCoGrInPrMe}.

\subsection{Existence of Exotic Ideals}

Let $(X,\Gamma,\alpha)$ be a free transformation group. If $\Gamma$ is amenable, then the full and reduced norms on $C_c(\Gamma,C(X))$ coincide, so $\nXG$ is a singleton set. In that case, the unique represented free transformation group, $C(X) \rtimes_f \Gamma$, contains no exotic ideals. In the non-amenable case, we have the following result:

\begin{theorem} \label{exotic ideals exist}
For any countably-infinite, non-amenable discrete group $\Gamma$, there exists a free action of $\Gamma$ on a compact metric space $X$ such that $C(X) \rtimes_f \Gamma$ has an exotic ideal.
\end{theorem}

\begin{proof}
  By \cite[Theorem 1]{ElekFrMiAcCoGrInPrMe}, there exists a free and
  minimal action of $\Gamma$ on the Cantor set $X$ which admits an
  invariant (regular) Borel probability measure $\mu$. By Theorem
  \ref{amenable char}, the full and reduced norms on
  $C_c(\Gamma,C(X))$ do not coincide, so $C(X) \rtimes_f \Gamma$ has
  an exotic ideal, namely the kernel of the canonical $*$-epimorphism,
  $\ker\Lambda_f$.
\end{proof}

\begin{remark}{Remark} \label{frnam} For general $\Gamma$, the proof
  of Theorem \ref{exotic ideals exist} relies on
  \cite[Theorem~1]{ElekFrMiAcCoGrInPrMe} to furnish a free action on a
  compact metric space $X$ which admits an invariant Borel probability
  measure $\mu$. In some cases, for example when $\Gamma = \bbF_2$, such an
  action can be described explicitly and we do so now. Within the group $\so$, let
  $R_z$ be rotation by $\theta = \arccos(1/3)$ about the $z$-axis, and
  let $R_x$ be rotation by $\theta$ about the $x$-axis. It is
  well-known that the subgroup of $\so$ generated by $R_z$ and $R_x$
  is isomorphic to $\bbF_2$, the free group on
  generators.\footnote{Incidentally, this is the subgroup of $\so$
    that was used by Banach and Tarski to prove their well-known
    ``paradox''.} Identifying $\bbF_2$ with this subgroup, we may view
  it as acting on $X := \so$ by left translation. Normalized Haar
  measure on $\so$ is then the desired probability measure.
 \end{remark}

\subsection{Abundance of Exotic Ideals}\label{AbExId}

Let $\Gamma = \bbF_2$, $X = \so$, and $\Gamma \curvearrowright X$ be
the action described in Remark \ref{frnam}. In this setting, we prove
that the inclusion $(C(X) \rtimes_f \Gamma, C(X))$ has uncountably
many distinct exotic ideals. For this, it is clearly enough to produce
uncountably many distinct $C^*$-norms on $C_c(\Gamma,C(X))$. We will
do so using the theory of exotic group algebras and exotic
crossed-products introduced by Brown and Guentner in
\cite{BrownGuentnerNeC*CoDiReSp}. This theory was further developed by
Kaliszewski, Landstad, and Quigg in
\cite{KaliszewskiLandstadQuiggExGrC*AlNoDu}, and also by Buss,
Echterhoff, and Willett in \cite{BussEcterhoffWillettExCrPrBaCoCo}.

Here is a summary of the background material on exotic group algebras
and crossed products we require.  For the moment, we do not assume the
\cstaralg\ $\A$ is abelian.
Recall (e.g., from \cite{BussEcterhoffWillettExCrPrBaCoCo}) that given a discrete (for simplicity) group $\Gamma$, a \emph{crossed-product functor} is a functor $\E$ from the category of $\Gamma$-\cstar-algebras (that is, \cstar-algebras equipped with an action of $\Gamma$) to the category of \cstar-algebras. In symbols, this is usually expressed by
\[
  (\A,\tau) \overset{\E}{\mapsto} \A \rtimes_{\tau,\E} \Gamma,
\]
where $\tau$ is the designated action of $\Gamma$ on $\A$. In addition, it is generally assumed that every $\A \rtimes_{\tau,\E} \Gamma$ is a completion of $C_c(\Gamma,\A)$ under a \cstar-norm $\eta_\E$ lying between the full and reduced norms, and moreover that for every $\Gamma$-covariant morphism
\begin{equation}\label{uphi}
    \varphi: (\A,\tau) \rightarrow (B,\beta),
\end{equation}
the associated morphism
\[
  \E(\varphi): \A \rtimes_{\tau,\E} \Gamma \rightarrow B \rtimes_{\tau,\E} \Gamma
\]
maps the dense copy of $C_c(\Gamma,\A)$ into $C_c(\Gamma,B)$ in the expected way, namely, sending a function $f \in C_c(\Gamma,\A)$ to $\varphi \circ f \in C_c(\Gamma,B)$.\\

From a more pedestrian point of view, a crossed-product functor simply boils down to choosing a \cstar-norm $\eta_\E$ satisfying
\begin{equation} \label{ufact.c}
    \|x\|_r \leq \eta_\E(x) \leq \|x\|_f
\end{equation}
on $C_c(\Gamma,\A)$, for every $\Gamma$-\cstar algebra $(\A,\tau)$, and such that, for every $\Gamma$-covariant $*$-homomorphism $\varphi: (\A,\tau) \to (B,\beta)$, the corresponding map
\begin{equation} \label{CoreFunctor}
    f \in C_c(\Gamma,\A) \mapsto \varphi \circ f \in C_c(\Gamma,B),
\end{equation}
is continuous for the chosen norms.\\

Observe that, if the map $\varphi $ of \eqref{uphi} is a completely positive, covariant map, rather than a $*$-homomorphism, then \eqref{CoreFunctor} still gives a well-defined map. However, there is no guarantee that \eqref{CoreFunctor} is continuous for the chosen norms. In fact, continuity fails for certain crossed-product functors, such as most Brown--Guentner functors (see \cite[Remark 4.19]{BussEcterhoffWillettExCrPrBaCoCo}). On the bright side, the crossed-product functors introduced by Kaliszewski, Landstad, and Quigg, usually referred to as KLQ-functors, all have the \emph{cp-map property}, meaning that \eqref{CoreFunctor} is always continuous and hence extends to a completely positive map between the respective completions (see~\cite[Corollary 4.18 and Theorem 4.9]{BussEcterhoffWillettExCrPrBaCoCo}).\\

The KLQ-functors form a class of crossed-product functors which can be put in a one-to-one correspondence with the set of all weakly closed, $\Gamma$-invariant (for left and right translations) ideals of the Fourier-Stieltjes algebra $B(\Gamma)$ \cite[Proposition 23]{BussEcterhoffWillettExCrPrBaCoCo}.The reader is referred to \cite{KaliszewskiLandstadQuiggExGrC*AlNoDu}
for more details on the KLQ-functors. A fact about KLQ-functors important to us is that if $E \idealin B(\Gamma)$ is a $\Gamma$-invariant ideal and $\E$ is the corresponding crossed-product functor, then $\bbC \rtimes_{\id,\E} \Gamma$ is the quotient of the full group \cstar-algebra $C_f^*(\Gamma)$ by the pre-annihilator of $E$ in $C_f^*(\Gamma)$, once the dual of $C^*_f(\Gamma)$ has been identified with $B(\Gamma)$, as usual. Among the $\Gamma$-invariant ideals of $B(\Gamma)$ are the ideals $E_p$, defined for each $p \in [1,\infty )$ as the weak$^*$-closed ideal of $B(\Gamma)$ generated by all positive-definite functions in $\ell^p(\Gamma)$.\\

In the fundamental paper \cite{OkayasuFrGrC*AlAsWielp}, Okayasu proved that, for the case of the free group on two generators $\bbF_2$, the KLQ-functors associated to the ideals $E_p$ assign pairwise distinct norms to $C_c(\bbF_2)$, for all $p \in [2,\infty)$.\\

Given any dynamical system $(X,\Gamma,\tau)$ admitting an invariant regular Borel probability measure $\mu$, such as the one described in Remark \ref{frnam}, observe that the map
\[
    \varphi: f \in C(X) \mapsto \int_X f\, d\mu \in \bbC
\]
is completely positive and $\Gamma$-covariant. Choosing any crossed-product functor $\E$ satisfying the cp-map property, we
therefore deduce the existence of a completely positive map
\[
    \Phi: C(X) \rtimes_{\tau,\E} \Gamma \rightarrow C^*_\E(\Gamma),
\]
where $C^*_\E(\Gamma )$ denotes $\bbC \rtimes_{\id,\E} \Gamma$, with ``id'' being the trivial action of $\Gamma$ on $\bbC$. On the other hand, relying only on the functoriality of $\E$, we also obtain a $*$-homomorphism
\[
    \iota: C^*_\E(\Gamma ) \to C(X) \rtimes_{\tau,\E} \Gamma,
\]
arising from the natural inclusion of $\bbC$ in the unital algebra $C(X)$, which is clearly a covariant map. Since $\Phi \circ \iota$ is the identity mapping on the dense subalgebra $C_c(\Gamma)$, we conclude that $\Phi \circ \iota$ is the identity map on $C^*_\E(\Gamma)$, whence $\iota $ is one-to-one and the range of $\iota$ is $*$-isomorphic to $C^*_\E(\Gamma)$. Observe that this may be reinterpreted as saying that the norm assigned by $\E$ on $C_c(\Gamma,C(X))$, if restricted to the subalgebra $C_c(\Gamma,\bbC) \simeq C_c(\Gamma)$, coincides with the norm assigned by $\E$ on the latter, once interpreted from the point of view of the dynamical system $(\bbC,\Gamma,\id)$.\\

The crucial conclusion is that, should $\E_1$ and $\E_2$ be two crossed-product functors satisfying the cp-map property, and such that the norms $\eta _{\E_1}$ and $\eta_{\E_2}$ on $C_c(\Gamma)$ do not coincide, then the two crossed-product functors will necessarily yield different norms on $C_c(\Gamma,C(X))$. With these preparations, we now have the following result.

\begin{theorem} \label{NegAnsQThom0}
Let $(X,\bbF_2,\alpha)$ be any free transformation group admitting an invariant probability measure $\mu$, such as the one described in Remark \ref{frnam}. Then there are uncountably many distinct \cstar-norms satisfying \eqref{ufact.c} on $C_c(\bbF_2,C(X))$.
\end{theorem}

\begin{proof}
Our discussion shows it is enough to prove that there are uncountably many KLQ-functors producing pairwise distinct norms on $C_c(\bbF_2)$. But this is precisely what has been shown by Okayasu in \cite{OkayasuFrGrC*AlAsWielp}, as mentioned above.
\end{proof}

\subsection{Exotic Ideals of Compact Operators}

We have seen in the previous sections that exotic ideals exist
(sometimes in abundance!), but the examples do not give explicit
elements belonging to exotic ideals. The purpose of this section is
to give examples of represented free transformation groups which
contain the compact operators as an exotic ideal.

\begin{theorem} \label{exotic compact}
Let $\Gamma$ be a countably-infinite discrete group with Kazhdan's property (T). Then there exists a represented free transformation group $C(X) \rtimes_\eta \Gamma \subseteq B(\H)$ containing the compact operators $\K(\H)$ as an exotic ideal.
Moreover, if $\{0\}\neq J\idealin C(X)\rtimes_\eta\Gamma$ is a proper  ideal, then
\begin{equation}\label{ExC1}
  \K(\H_\phi)\subseteq J\subseteq \L_\eta.
\end{equation}

\end{theorem}

\begin{proof}
Since $\Gamma$ is non-amenable, there exists a free and minimal action
of $\Gamma$ on the Cantor set $X$ which admits an ergodic (regular,
non-atomic) Borel probability measure $\mu$ \cite[Theorem
1]{ElekFrMiAcCoGrInPrMe}. For $f \in C(X)$, let $\phi(f) = \int_X
f(x)d\mu(x)$, an ergodic state. By Proposition \ref{Koopman facts},
the corresponding Koopman representation $\kappa_\phi:C_c(\Gamma,C(X))
\to B(\H_\phi)$ is injective. Thus the formula $\eta(x) =
\|\kappa_\phi(x)\|$ defines a $C^*$-norm on $C_c(\Gamma,C(X))$, and
the completion $C(X) \rtimes_\eta \Gamma$ can be identified with
$\kappa_\phi(C(X)\rtimes_f\Gamma)\subseteq B(\H_\phi)$. By
Theorem \ref{contains compacts}, $\K(\H_\phi) \subseteq C(X)
\rtimes_\eta \Gamma$. Since $\mu$ is non-atomic, $C(X) \cap
\K(\H_\phi) = \{0\}$, otherwise $C(X)$ would contain a non-zero
finite-rank projection, which is impossible.  Thus $\K(\H_\phi)$ is
an exotic ideal in $C(X)\rtimes_\eta\Gamma$.

Since the compact operators are an essential ideal of $\bh$, there is
no $0\neq T\in C(X)\rtimes_\eta\Gamma$ for which $T\K(\H_\phi)=\{0\}$ or
$\K(\H_\phi) T=\{0\}$.  Thus, the simplicity of $\K(\H_\phi)$ implies that
if $J\idealin C(X)\rtimes_\eta\Gamma$ is a non-zero proper ideal, then
$\K(\H_\phi)\subseteq J$.   Next,
the set $Z:=\{x\in X: f(x)=0 \text{ for every } f\in C(X)\cap J\}$ is a
closed set invariant under the action of $\Gamma$ on $X$.  By minimality of the action,
$Z$ is either empty or $X$.   If $Z=\emptyset$, then $C(X)\cap J=C(X)$, which
implies $J=C(X)\rtimes_\eta \Gamma$, which is impossible
as $J$ is a proper ideal.  Therefore $J\cap C(X)=\{0\}$, so $J$ is an
exotic ideal.   An application of Corollary~\ref{free
  facts}\eqref{ff6} gives~\eqref{ExC1}.

\end{proof}

\begin{remark}{Remark}
The proof of Theorem~\ref{exotic compact} gives somewhat
    more.
  \begin{enumerate}
  \item  If $C_c(\Gamma, C(X))$ is identified with its image under
    the Koopman representation, then
    \begin{equation}\label{ec1}
      C_c(\Gamma,C(X))\cap \K(\H_\phi)=\{0\}.
    \end{equation}
    To see this, note that \eqref{EFaith} shows that the conditional expectation
    $E_\eta:C(X)\rtimes_\eta\Gamma\rightarrow C(X)$ is faithful on
    $C_c(\Gamma,C(X))$.   The fact that $\K(\H_\phi)\subseteq \{x\in
    C(X)\rtimes_\eta\Gamma: E_\eta(x^*x)=0\}$ gives~\eqref{ec1}.
    \item The fact that $\norm{\cdot}_r\leq \eta$,  implies that the
      canonical surjection $\theta: C(X)\rtimes_\eta\Gamma\rightarrow
      C(X)\rtimes_r\Gamma$ satisfies
      \begin{equation}\label{cc2} C_c(\Gamma,C(X))\cap \ker\theta=\{0\}.
      \end{equation}
      \end{enumerate}
\end{remark}
We conclude this section with some questions.

\begin{remark}{Question} \label{mudepend}
It follows from \cite[Corollary~1.1]{ElekFrMiAcCoGrInPrMe} that the free dynamical system $(X,\Gamma,\alpha)$ may be chosen so that it admits multiple invariant ergodic Borel probability measures. Supposing that $\mu_1$ and $\mu_2$ are such measures, let $\eta_1$ and $\eta_2$ be the norms arising from the corresponding Koopman representations. Are the \cstaralg s $C(X) \rtimes_{\eta_1} \Gamma$ and $C(X) \rtimes_{\eta_2} \Gamma$ isomorphic?
\end{remark}

\begin{remark}{Question} \label{ExoticCartan}
In the context of Theorem \ref{exotic compact}, does $C(X) \rtimes_\eta \Gamma$ contain a Cartan MASA?
\end{remark}

\begin{remark}{Question}
  Again in the context of Theorem \ref{exotic compact}, does
  $C(X) \rtimes_\eta \Gamma$ contain a non-unitary isometry? We note
  that $C(X) \rtimes_r \Gamma$ has a faithful trace, so every isometry
  in $C(X) \rtimes_r \Gamma$ is unitary.
\end{remark}

\begin{remark}{Question} \label{KMax2?}  We continue with the setting of Theorem
  \ref{exotic compact}.
  Are the compact operators the only exotic
  ideal of $C(X) \rtimes_\eta \Gamma$?  In other words, we ask: is
  $\L_\eta=\K(\H_\phi)$?
\end{remark}

\section{A Nontrivial Opaque Ideal} \label{SecOpId}

We begin this section by briefly describing the opaque ideal, introduced in \cite[Definition 2.10.4]{ExelPittsChGrC*AlNoHaEtGr}.
Let $(\B,\A)$ be a regular inclusion with $\A$ abelian, and let $X$ be the spectrum of $\A$. Therefore $X$ is a compact Hausdorff space and $\A$ is $*$-isomorphic to $C(X)$. Hence we may, and will, identify $\A$ with $C(X)$ throughout.

For each $x$ in $X$, denote
\[
    J_x = \{f \in \A: f(x) = 0\};
\]
this is a maximal ideal of $\A$.

In what follows we will often refer to the set $J_x\B$. The reader may think of this as the closed linear subspace of $\B$ spanned by the set of products $\{ab: a\in J_x, b \in \B\}$. Actually, the Cohen-Hewitt Factorization Theorem \cite[32.22]{HewittRossAbHaAnII} shows that taking the closed linear span is unnecessary, meaning that
\[
    \{ab: a \in J_x, b \in \B\} = \overline\spn~\{ab: a \in J_x, b \in \B\}.
\]

By \cite[Proposition 2.10.4]{ExelPittsChGrC*AlNoHaEtGr} the set
\[
    \Delta := \bigcap_{x \in X} J_x\B
\]
is an ideal of $\B$, called the \emph{opaque ideal}, and moreover
\[
    \Delta = \bigcap _{x \in X} J_x\B = \bigcap_{x \in X} \B J_x.
\]
By \cite[Proposition 2.10.11]{ExelPittsChGrC*AlNoHaEtGr}, $\Delta$ has trivial intersection with $\A$, so it is an example of an exotic ideal whenever it is non-zero.

To get a feeling for the opaque ideal, suppose that $\B$ is abelian, and hence $\B=C(Y)$ for some compact Hausdorff space $Y$. In addition there is a continuous surjective map
\[
    j:Y \to X
\]
describing the inclusion of $\A$ in $\B$ (up to the relevant isomorphisms) by means of the formula
\[
    f \in C(X) \mapsto f \circ j \in C(Y), \qquad (f \in C(X)).
\]
It is not hard to see that, given $x$ in $X$, the set $J_x\B$ turns out to be the ideal formed by the functions in $C(Y)$ vanishing on the fibre $j^{-1}(\{x\})$. Thus, if $f$ lies in $\Delta $, we have that $f$
vanishes on $j^{-1}(\{x\})$, for every $x$ in $X$, but since
\[
    Y = \bigcup _{x\in X} j^{-1}(\{x\}),
\]
one sees that $f$ vanishes everywhere on $Y$. In other words, the commutativity of $\B$ implies that $\Delta = \{0\}$.

Besides the abelian case, there are many other settings for which $\Delta =\{0\}$, so the non-vanishing of the opaque ideal should perhaps be thought of as an anomaly. Nevertheless, our main goal in this section to exhibit precisely such an anomaly, that is, a regular inclusion admitting a nonzero opaque ideal.

There is another relevant ideal of $\B$, called the \emph{gray ideal}, defined by
\[
    \Gamma := \bigcap_{x \in F} J_x\B,
\]
where $F$ is the set of points $x$ in $X$ (known as the \emph{free points}), for which the evaluation functional
\[
    \text{ev}_x: f\in C(X) \mapsto f(x) \in {\mathbb C},
\]
admits a unique pure state extension to $\B$. Just like the opaque ideal, one also has that
\[
    \Gamma = \bigcap_{x \in F} J_x\B = \bigcap_{x \in F} \B J_x.
\]
See \cite[Proposition 2.10.8]{ExelPittsChGrC*AlNoHaEtGr} for more details.

Under the assumption that every point of $X$ is free, one clearly has that $\Gamma = \Delta$. This situation is exemplified by the inclusion of $C(X)$ into a represented free transformation group $C(X) \rtimes_\eta G$. The promised example may now be described:

\begin{theorem} \label{non-trivial-opaque-ideal}
  Let $G $ be a non-amenable discrete group and let
  $(X, G , \alpha )$ be any freely acting discrete dynamical
  system admitting an invariant probability measure, such as
  Example~\ref{frnam},  above.  Then the opaque and gray ideals relative to
  the inclusion
  $$
  \big (C(X)\rtimes _f G ,C(X)\big )
  $$
  coincide and are non-zero.
  \end{theorem}

\begin{proof}
  As already noticed $\Delta =\Gamma $, so it suffices to prove that
$\Gamma \neq \{0\}$.

  By Theorem \ref{amenable char} we have that the full and reduced norms on $C_c(G,C(X))$ do not coincide,   as
already argued in the proof of Theorem \ref{exotic ideals exist}, so
the kernel of the  canonical
  $*$-epimorphism
  $\Lambda_f:C(X) \rtimes_f G \twoheadrightarrow C(X) \rtimes_r G$
  is nontrivial.
  Employing \cite[Proposition 3.5.9(i)]{ExelPittsChGrC*AlNoHaEtGr} we then have that
  $$
  \Gamma =\ker(\Lambda_f)\neq \{0\},
  $$
  concluding the proof.
  \end{proof}

It seems to us that non-amenability is the key ingredient in the above result, so one might wonder about the existence of examples not relying on the lack of amenability. Of course not all regular inclusions arise from groups acting on topological spaces so, in order to make sense of the previous sentence, we need to replace amenability by some condition that makes sense for inclusions not arising from group actions. Based on \cite[Theorem 2.6.8]{BrownOzawaC*AlFiDiAp} it seems sensible to replace amenability with nuclearity, so we ask:

\begin{question}\label{NRQ}
Let $(\B,\A)$ be a regular inclusion with $\A$ abelian and $\B$ nuclear. Must the associated opaque ideal vanish?
\end{question}

Under the extra hypothesis that $(\B,\A)$ has the extension property,
Question~\ref{NRQ} has been answered affirmatively in~\cite[Proposition~4.2]{ExelOnKuC*DiOpId}

\section{Applications}

The existence of exotic ideals in represented free transformation groups can be used to give examples of phenomena and
answer other questions in the literature. The purpose of this section is to highlight three additional applications along these lines.

\subsection{Thomsen's Questions} \label{ThomsenSect}

In \cite[Remark~16]{ThomsenOnFrTrGpC*Al}, Thomsen asks the following two questions, which at least partially motivated our study in the first place:

\begin{question} \label{TQuests}
\begin{enumerate}
\item\label{TQuest0} Does there exist a represented free transformation group which does not arise from a full or reduced crossed product?
\item\label{TQuest} Is the conditional expectation of a represented free transformation group always faithful?
\end{enumerate}
\end{question}

Theorem \ref{NegAnsQThom0} answers Question \ref{TQuest0} affirmatively, which in turn answers Question \ref{TQuest} negatively. Alternatively, Theorem \ref{exotic ideals exist} emphatically answers Question \ref{TQuest} (negatively, of course).

\subsection{Katavolos and Paulsen's Question} \label{KPSect}

Suppose $\H$ is a separable Hilbert space and $\fA$ is a MASA in $\bh$. Katavolos and Paulsen \cite{KatavolosPaulsenOnRaBiPr} studied ranges of certain idempotent $\fA$-bimodule mappings of $\bh$ into itself. The norm-closed linear span of the normalizers of $\fA$ is a \cstaralg\ $\fB$, and Katavolos and Paulsen show that the collection of bimodule maps they consider leave $\fA$, $\fB$, and $\fB+\K$ invariant; furthermore, the restrictions of the bimodule maps to these algebras form a commutative algebra \cite[Proposition 15]{KatavolosPaulsenOnRaBiPr}. As part of their work, in the case when $\fA$ is a non-atomic MASA, they ask whether $\fB$ contains the compact operators $\K(\H)$ \cite[Remark~13]{KatavolosPaulsenOnRaBiPr}). Just before \cite[Proposition~15]{KatavolosPaulsenOnRaBiPr}, they  implicitly conjecture that it does. The purpose of this section is to confirm their intuition. (It then follows readily that the inclusion $\K\subseteq \fB$ still holds if the non-atomic MASA $\fA$ is replaced with any MASA in $\bh$.) Represented free transformation groups will be a tool for this.\\

\begin{corollary}[of Theorem \ref{exotic compact}] \label{inbh}
Let $\H$ be a separable, infinite-dimensional Hilbert space, and let $\fA\subseteq \bh$ be a non-atomic MASA. Put $\fB_0=\spn~\N(\bh,\fA)$ and $\fB=\overline\spn~\N(\bh,\fA)$. Then the following statements hold:
\begin{enumerate}
\item $\fB$ contains the compact operators $\K(\H)$, yet $\fB_0 \cap \K(\H) = \{0\}$.
\item $(\fB,\fA)$ is a regular MASA inclusion having the extension property.
\item There is a unique conditional expectation $E:\fB \rightarrow \fA$, and $E$ annihilates $\K(\H)$.
\end{enumerate}
\end{corollary}

\begin{proof}
Keep the notation from Theorem \ref{exotic compact} and its proof. Since any two non-atomic MASAs on a separable Hilbert space are unitarily equivalent, we may assume that $\fA = \kappa_\phi(C(X))'' \subseteq B(\H_\phi)$. For every $s \in \Gamma$, $\kappa_\phi(s)\kappa_\phi(C(X))\kappa_\phi(s)^* = \kappa_\phi(C(X))$; it follows that $\fA$ is invariant under $\Ad_{\kappa_\phi(s)}$. Therefore $\K(\H_\phi) \subseteq C(X) \rtimes_\eta \Gamma \subseteq \fB$. By \cite[Theorem~2.21]{PittsStReInI}, $(\fB,\fA)$ is a regular MASA inclusion with the extension property, so there is a unique conditional expectation $E:\fB \rightarrow \fA$. Since $\fA$ has no atoms, it cannot contain any finite-rank
projections, and therefore cannot contain any non-zero compact operator. Thus \cite[Theorem~3.15]{PittsStReInI} shows $\K(\H_\phi)$ is contained in the ideal $\L := \{x \in \fB: E(x^*x) = 0\}$, and therefore $E$ annihilates $\K(\H_\phi)$. An
application of \cite[Lemma~7.3]{PittsStReInI} yields $\fB_0 \cap \K(\H_\phi) = \{0\}$.
\end{proof}

\begin{remark}{Remark}\label{useunitary}
The \cstaralg\ $\fB$ found in Corollary~\ref{inbh} can also be described as the closed linear span of the \textit{unitary} normalizers of $\fA$. This follows from the polar decomposition and the proof of \cite[p.~479, inclusion~(2.8)]{CameronPittsZarikianBiCaMASAvNAlNoAlMeTh}. Write
$\fA \simeq L^\infty(\Omega,\mu)$ and take $\Gamma_\fA$ to be the group of \textit{all} measure preserving transformations of $(\Omega,\mu)$. Elements of this group correspond under the Koopman representation to unitary normalizers of $\fA$. Thus $\fB$ is a represented transformation group, although the action of $\Gamma_\fA$ on $\widehat{L^\infty(\Omega,\mu)}$ is not free.
\end{remark}

\begin{remark}{Question}
For the inclusion $(\fB,\fA)$ of Corollary \ref{inbh}, is $\K(\H_\phi)$ the left kernel of the conditional expectation $E:\fB \rightarrow \fA$?
\end{remark}

\subsection{Regularity is Not Hereditary from Above}

Consider three unital \cstaralg s with $\A \subseteq \B \subseteq \C$. To say \emph{regularity is hereditary from above} means that regularity of $(\C,\A)$ implies regularity of $(\B,\A)$. Whether or not regularity is hereditary from above was among a list of open problems in \cite{PittsZarikianUnPsExC*In}.\\

In general, regularity is not hereditary from above, as can be seen from \cite[Example~5.1]{BrownExelFullerPittsReznikoffInC*AlCaEm}. Indeed, that example gives an inclusion $(\C,\A)$ in which $\A$ is a Cartan MASA in $\C$ (so $(\C,\A)$ is regular) and an intermediate subalgebra $\B$ for which $\A$ is not Cartan in $\B$. Note that the restriction $E|_\B$ of the conditional expectation $E:\C \rightarrow \A$ is a faithful conditional expectation of $\B$ onto $\A$, and since $\A$ is a MASA in $\C$, $\A$ is necessarily a MASA in $\B$. Thus $(\B,\A)$ cannot be regular, for if it were, $\A$ would also be Cartan in $\B$.\\

We will use Proposition~\ref{regnotHerid} below and the existence of exotic ideals in represented free transformation groups to produce other examples showing regularity is not hereditary from above. These examples differ from \cite[Example~5.1]{BrownExelFullerPittsReznikoffInC*AlCaEm} in at least three ways: they have the extension property, they have no \textit{faithful} conditional expectation present, and the intermediate inclusion is singular.

\begin{proposition}\label{regnotHerid}
Suppose $(\C,\D)$ is a unital MASA inclusion. If $\{0\} \neq J \idealin \C$ satisfies $J \cap \D = \{0\}$ (i.e., if $J$ is an exotic ideal), then $(\D+J,\D)$ is a singular inclusion.
\end{proposition}

\begin{proof}
Let $v$ be an intertwiner for $(\D+J,\D)$ and write $v=d+j$ with $d \in \D$ and $j \in J$. Then for any $h \in \D$, there exists $h' \in \D$ such that $vh=h'v$; that is, $dh+jh=h'd+h'j$. This gives $dh-h'd=h'j-jh$. As $J \cap \D = \{0\}$, $h'j=jh$. It follows that $j\D \subseteq \D j$. Similarly, $\D j \subseteq j\D$, so $j \in \INT(\D+J,\D)$. Since $\D$ is a MASA in $\D+J$,~\cite[Proposition~3.3]{DonsigPittsCoSyBoIs} shows $j$ is a normalizer. Thus $j\D j^* \subseteq J \cap \D$. Since
$1 \in \D$, we obtain $jj^*=0$, that is, $j=0$. This shows that $\INT(\D+J,\D)=\D$. By \cite[Proposition~3.4]{DonsigPittsCoSyBoIs}, $\D=\overline{\INT(\D+J,\D)}=\N(\D+J,\D)$. Thus, $(\D+J,\D)$ is a singular inclusion.
\end{proof}

\begin{corollary}
If $C(X) \rtimes_\eta \Gamma$ is a represented free transformation group with an exotic ideal $J$, then $(C(X) \rtimes_\eta \Gamma,C(X))$ is regular MASA inclusion with the extension property and a unique conditional expectation which is not faithful, but $(C(X)+J,C(X))$ is a singular inclusion.
\end{corollary}


\def\cprime{$'$}
\providecommand{\bysame}{\leavevmode\hbox to3em{\hrulefill}\thinspace}
\providecommand{\MR}{\relax\ifhmode\unskip\space\fi MR }
\providecommand{\MRhref}[2]{%
  \href{http://www.ams.org/mathscinet-getitem?mr=#1}{#2}
}
\providecommand{\href}[2]{#2}

\end{document}